\documentclass[12pt,draft]{amsart}
\usepackage[all]{xy}
\usepackage{amsfonts}
\usepackage{amssymb}

\usepackage{enumerate}

\usepackage{color}

\renewcommand{\le}{\leqslant}

\newtheorem{teo}{Theorem}[section]
\newtheorem{lem}[teo]{Lemma}
\newtheorem{prop}[teo]{Proposition}

\newtheorem{cor}[teo]{Corollary}

\theoremstyle{definition}
\newtheorem{dfn}[teo]{Definition}
\newtheorem{rk}[teo]{Remark}
\newtheorem{ex}[teo]{Example}

\def\<{\langle}
\def\>{\rangle}
\def\ss{\subset}

\def\r{\rho}

\def\t{\tau}

\def\f{{\varphi}}

\def\G{{\Gamma}}

\def\C{{\mathbb C}}

\def\Z{{\mathbb Z}}

\def\End{\mathop{\rm End}\nolimits}
\def\Ker{\mathop{\rm Ker}\nolimits}
\def\Im{\mathop{\rm Im}\nolimits}

\def\Id{\operatorname{Id}}

\def\Tr{\operatorname{Tr}}

\def\1{\mathbf 1}

\newcommand{\til}[1]{\widetilde{#1}}
\newcommand{\wh}[1]{\widehat{#1}}

\def\N{{\mathbb N}}

\begin{document}

\title[TBFT for endomorphisms]
{Twisted Burnside-Frobenius theory
for endomorphisms of polycyclic groups}

\author{Alexander Fel'shtyn}
\thanks{The work of A.F. is partially supported by the  grant  346300 for IMPAN from the Simons Foundation and the matching 2015-2019 Polish MNiSW fund.}
\address{Instytut Matematyki, Uniwersytet Szczecinski,
ul. Wielkopolska 15, 70-451 Szczecin, Poland} \email{felshtyn@mpim-bonn.mpg.de,
 fels@wmf.univ.szczecin.pl}

\author{Evgenij Troitsky}
\thanks{The work of E.T. is partially
supported by the Russian Foundation for Basic
Research under grant 16-01-00357.}
\address{Dept. of Mech. and Math., Moscow State University,
119991 GSP-1  Moscow, Russia}
\email{troitsky@mech.math.msu.su}
\urladdr{
http://mech.math.msu.su/\~{}troitsky}


\begin{abstract}
Let $R(\f)$ be the number of $\f$-conjugacy (or Reidemeister)
classes of an endomorphism $\f$
of a group $G$.
We prove for several classes  of groups (including polycyclic)
that the number $R(\f)$
is equal to the number of fixed points of the induced
map of an appropriate subspace of the unitary dual space
$\wh G$,
when $R(\f)<\infty$.
Applying the result to iterations of $\f$ we obtain Gauss 
congruences for Reidemeister numbers.

In contrast with the case of automorphisms, studied previously,
we have a plenty of examples having the above finiteness condition, even among groups with
$R_\infty$ property. 
\end{abstract}

\maketitle

\section*{Introduction}
The \emph{Reidemeister number} or $\f$-\emph{conjugacy number}
of an endomorphism $\f$ of a group $G$ is the number of
its \emph{Reidemeister} or $\f$-\emph{conjugacy classes},
defined by the equivalence
$$
g\sim x g \f(x^{-1}).
$$

The interest in twisted conjugacy relations has its origins, in
particular, in the Nielsen-Reidemeister fixed point theory (see,
e.g. \cite{Jiang,Jiang84,FelshB}), in Arthur-Selberg theory (see, e.g.
\cite{Shokra,Arthur}),  Algebraic Geometry (see, e.g.
\cite{Groth}), and Galois cohomology (see, e.g.
\cite{SerreGaCohom}). In representation theory twisted conjugacy probably
occurs first in \cite{Gantmacher} (see, e.g.
\cite{Springer,Onishik-Vinberg}).

An important problem in the field is to identify the 
Reidemeister numbers with numbers 
of fixed points on an appropriate space in a way respecting
iterations. This opens possibility of obtaining congruences
for Reidemeister numbers and other important information.

For the role of the above 
``appropriate space'' typically some versions
of unitary dual can be taken.
This desired construction is called the twisted Burnside-Frobenius
theory (TBFT), because in the case of a finite group and
identity automorphism we arrive to the classical Burnside-Frobenius
theorem on enumerating of (usual) conjugacy classes. 

In the present paper we prove TBFT for endomorphisms of
any polycyclic group (Theorem \ref{teo:polyend}) and Gauss congruences for Reidemeister numbers.

In the case of automorphism this problem was solved
for polycyclic-by-finite groups in \cite{polyc,feltroKT}.
Preliminary and related results, examples and counter-examples
can be found in \cite{FelHill,FelHillForum,FelshB,FelTro,FelTroVer,FelIndTro,FelTroLuch,ncrmkwb,TroTwoEx}.

For endomorphisms of polycyclic fundamental groups
of infra-solvmanifolds of type ({R}) the Gauss congruences
were obtained in \cite{FelLee2015}.

The importance of obtaining the present results namely
for endomorphisms is justified by a plenty of examples
(in contrast with the case of automorphisms) (see,
in particular,
\ref{prop:drast}, \ref{cor:finiteimendo}, and
\ref{examp:halee} below).

A related fact is that,
in contrast with TBFT for automorphisms, TBFT for
endomorphisms is weakly connected with the theory
of $R_\infty$-groups (see e.g. Example \ref{examp:halee}).
A group is called $R_\infty$ if any its \emph{automorphism}
has infinite Reidemeister number. This was the subject
of an intensive recent research and for many groups this
property was established, see the following partial
bibliography and the literature
therein: \cite{FelPOMI,ll,f07,FelGon08Progress,%
FelLeonTro,bfg,GK,dfg,%
Nasybull2012,MubeenaSankaran2014TrGr,%
FelGon2011Q,GoWon09Crelle,%
Romankov,DekimpeGoncalves2014BLMS,FelNasy2016JGT,%
HaLee2015}.
In some situations the property $R_\infty$ has some direct topological
consequences (see e.g. \cite{GoWon09Crelle}). 

Decision problems for twisted conjugacy classes of endomorphisms of polycyclic groups
were studied in \cite{Romankov2010JGT}.

\medskip
The paper is organized in the following way.

In Section \ref{sec:prelimi} we show how drastically the theory of twisted conjugacy classes for endomorphisms differs from the theory for automorphisms.

In Section \ref{sec:dualobjend} we introduce and investigate a dual object for a pair $(G,\f)$.

In Section \ref{sec:tbftcongr}  we prove Gauss congruences for Reidemeister numbers.

In Section \ref{sec:tbftabfin}  we discuss proof of twisted Burnside--Frobenius theorem 
for endomorphisms of any finite group.

Section \ref{sec:techlem} is technical.

In Section \ref{sec:tbftpolyc} we proof the twisted Burnside--Frobenius theorem for endomorphisms  of polycyclic groups.
We finish the paper with a series of examples of $R_\infty$
groups admitting endomorphisms with finite Reidemeister numbers .

\medskip
\textsc{Acknowledgement:} 
This work is a part of our joint research project
at the Max-Planck Institute for
Mathematics (Bonn) and the most part of the results were
obtained there in Spring 2017.

The work of A.F. is partially supported by the  grant  346300 for IMPAN from the Simons Foundation and the matching 2015-2019 Polish MNiSW fund.

The work of E.T. is partially
supported by the Russian Foundation for Basic
Research under grant 16-01-00357.

\section{Preliminaries}\label{sec:prelimi}

First of all, let us make the following observation, showing
how drastically the Reidemeister numbers world for endomorphisms
differs from the Reidemeister world for automorphisms.

\begin{prop}\label{prop:drast} 
For any group $G$ there exists an endomorphism $\f:G\to G$
with $R(\f)<\infty$, namely $R(\f)=1$.
\end{prop}

\begin{proof}
Take $\f$ to be the trivial map $\f(g)=e$ for any $g\in G$.
\end{proof}

This observation can be enforced.

\begin{prop}\label{prop:cosetsendo}
Suppose, $\f:G\to G$ is an endomorphism and $K:=\Ker\f$.
Then all Reidemeister classes are some unions of $K$-cosets. 
\end{prop}

\begin{proof}
Let $g_1$ and $g_2$ be in a $K$-coset, i.e. $g_1g_2^{-1}=k\in K$.
Then $g_1=kg_2=k g_2 \f(k^{-1})$.
\end{proof}

Using Lemma \ref{lem:epimorphmappingofclasses} we immediately obtain

\begin{cor}\label{cor:fififi}
The map $p_\f:G\to G/K$ gives a bijection of Reidemeister classes.
\end{cor}

\begin{cor}\label{cor:finiteimendo}
Any endomorphism with finite image has a finite Reidemeister
number.
\end{cor}

\begin{dfn}
Denote by $\wh G$ the \emph{unitary dual} of $G$, by
$\wh G_f$ the part of the unitary dual formed by
irreducible finite-dimensional representations,
and by $\wh G_{ff}$ the part of $\wh G_{f}$ formed by
 \emph{finite} representations, i.e. representations
 that factorize through a finite group.
\end{dfn}

\begin{dfn}
Let us call \emph{o.t. commutant} the operator theoretical
commutant of a set $D$ of bounded operators on a Hilbert space, i.e.
the subset of
the algebra of all bounded operators on this space, 
formed by all operators
that commute with all elements of $D$. Denote it $D^\bigstar$.
\end{dfn}

\begin{lem}\label{lem:irredim}
A representation is irreducible if and only if
the o.t. commutant of the set of representing operators
is formed just by scalar operators \cite[Theorem 2, p. 114]{Kirillov}.

In particular, if $\rho$ is irreducible and $\f$ is an
epimorphism, $\rho\circ\f$ is also irreducible. 
\end{lem}

\begin{lem}\label{lem:sopriazhobr}
If representations $\pi$ and $\rho$ of $G$ are equivalent,
then $\pi\circ\f$ and $\rho\circ\f$ are equivalent for
any endomorphism $\f:G\to G$.
\end{lem}

\begin{proof}
Indeed: use the same intertwining operator.
\end{proof}

The following statement is well known 
\begin{lem}\label{lem:epimorphmappingofclasses}
Suppose, $\f:G\to G$ is an endomorphism and
$H\ss G$ is a normal $\f$-invariant subgroup,
then $p:G\to G/H$ induces an epimorphism of Reidemeister
classes.
\end{lem}

\begin{proof}
Indeed, suppose, $p(g')=p(\til g) p(g) p(\f(\til {g}^{-1}))$. 
Then it is equal to $p(\til g g \f(\til {g}^{-1}))$.
\end{proof}

Also, we need the following
\begin{lem}\label{lem:fiinva}
Any Reidemeister class of $\f$ is $\f$-invariant.
\end{lem}

\begin{proof}
Indeed, $\f(x)=x^{-1} x \f(x)$.
\end{proof}

The following fact can be extracted from \cite[Prop. 1.6]{go:nil1}.
\begin{lem}\label{lem:fixedpointsonfaxtorandreid}
In the above situation $R(\f|_H)\le R(\f)\cdot |C(\f_{G/H})|$
where $C(\f_{G/H})$ is the fixed point subgroup for the
induced map $\f_{G/H}:G/H\to G/H$. 
\end{lem}

The following statement is well known in the field. 
\begin{lem}\label{lem:shifts}
A right shift by $g\in G$ maps bijectively Reidemeister classes of $\f$
onto Reidemeister classes of $\t_{g^{-1}}\circ \f$, where
$\t_g$ is the inner automorphism: $\t_g(x)=gxg^{-1}$.
\end{lem} 

\begin{proof}
This follows immediately from the equality
$$
xy\f(x^{-1})g=x (yg) g^{-1}\f(x^{-1})g=x (yg)(\t_{g^{-1}}\circ\f)(x^{-1}).
$$
\end{proof}

\section{Dual object for a pair $(G,\f)$}\label{sec:dualobjend}

For an \emph{automorphism} or at least \emph{epimorphism} $\f$
and an irreducible representation $\rho$, the representation
$\rho\circ\f$ is always irreducible. For a general endomorphism
and a finite-dimensional $\rho$ one can only decompose the
representation into irreducible components. So, one obtains a
sort of multi-valued mapping and the corresponding dynamical system
is not defined. To avoid this problem we will
act as follows. 

\begin{dfn}
We will call a representation class $[\rho]$ 
a $\wh\f$-\textbf{f}-point,
if $\rho$ is equivalent to $\rho\circ\f$ (we avoid to say
that it is a fixed point, because we can not define the
corresponding dynamical system).
\end{dfn}

\begin{dfn}
An element $[\rho]\in \wh{G}$ (respectively, in $\wh{G}_f$
or $\wh{G}_{ff}$)
is called $\f$-\emph{irreducible} if $\r\circ \f^n$ is
irreducible for any $n=0,1,2,\dots$.

Denote the corresponding subspaces of $\wh{G}$ (resp., $\wh{G}_f$
or $\wh{G}_{ff}$)
by $\wh{G}^\f$ (resp., $\wh{G}^\f_f$ or $\wh{G}^\f_{ff}$). 
\end{dfn}

In some important cases these subspaces coincide with
the entire spaces:

\begin{prop}\label{prop:goodex}
\begin{enumerate}
\item If $G$ is abelian, $\wh{G}=\wh{G}_f=\wh{G}^\f=\wh{G}^\f_f$.
\item If $\f$ is an epimorphism, $\wh{G}=\wh{G}^\f$ and $\wh{G}_f=\wh{G}^\f_f$
\end{enumerate}
\end{prop}

\begin{proof}
The first statement immediately follows from the fact that a
representation of an abelian group is irreducible if and only
if it is 1-dimensional.

The second one follows from Lemma \ref{lem:irredim} keeping in
mind that $\f^n$ is an epimorphism if and only if $\f$ is. 
\end{proof}

Evidently continuous (w.r.t. the topology of weak containment)
maps $\wh{\f^n}:[\rho]\mapsto[\rho\circ\f^n]$ are defined for these subspaces $\wh{G}^\f$, $\wh{G}_f^\f$, and $\wh{G}_{ff}^\f$
(generally not a homeomorphism!) and $\wh{\f^n}=(\wh{\f})^n$.
Thus we obtain a dynamical system as the corresponding 
action of the semigroup $\mathbb{N}_0=\{0,1,2,\dots\}$
(we will reserve $\mathbb{N}$ for $\{1,2,\dots\}$). 

The key observation is the following one:
\begin{lem}\label{lem:keylem}
Let $[\rho]\in \wh G$ be an $\wh\f^n$-{\rm\textbf{f}}-point for some $n\geq 1$,
i.e. the representations $\rho$ and $\rho\circ\f^n$ are
equivalent. Then $[\rho]\in \wh G^\f$.
\end{lem}

\begin{proof}
By Lemma \ref{lem:sopriazhobr} we obtain
$$
\rho\sim \rho\circ\f^n \sim \dots \sim \rho\circ\f^{kn} \sim \dots 
$$
In particular, all they are irreducible. Now consider an arbitrary
$m$ and choose $k$ such that $m \leq kn$.
Then $\Im \rho\circ\f^{m}\supseteq \Im \rho\circ\f^{kn}$
and $(\Im \rho\circ\f^{m})^\bigstar\subseteq (\Im \rho\circ\f^{kn})^\bigstar$.
It remains to apply Lemma \ref{lem:irredim} and conclude that
$\rho\circ\f^{m}$ is irreducible.
\end{proof}

\begin{cor}\label{cor:periodicanddyn}
So, we have no dynamical system generated by $\f$
on $\wh G$ (resp, $\wh G_f$, or $\wh{G}_{ff}$)
generally, but
we have the well-defined notion of a
$\wh\f^n$-\textbf{f}-point.

The corresponding well-defined dynamical system exists on 
$\wh G^\f$ (resp, $\wh G^\f_f$, or $\wh{G}^\f_{ff}$) and its $n$-periodic
points are exactly $\wh\f^n$-\textbf{f}-points.
\end{cor}

\begin{dfn}
Denote the number of $\wh\f^n$-\textbf{f}-points by $\mathbf{F}(\wh\f^n)$.
\end{dfn}

\section{TBFT implies congruences}\label{sec:tbftcongr}
\begin{dfn}\label{dfn:tbft}
We say that TBFT (resp., TBF$T_f$, TBF$T_{ff}$) takes place for
an endomorphism $\f:G\to G$ and its iterations, if
$R(\f^n)<\infty$ and $R(\f^n)$ coincides with the
number of $\wh\f^n$-\textbf{f}-points in $\wh G$
(resp., in $\wh G_f$, $\wh G_{ff}$)
for all $n\in \mathbb{N}$.

Similarly, one can give a definition for a single
endomorphism (without iterations).
\end{dfn}

\begin{dfn}
Denote by $\mu(d)$, $d\in\N$, be the {\em M\"obius function},
i.e.
$$
\mu(d) =
\left\{
\begin{array}{ll}
1 & {\rm if}\ d=1,  \\
(-1)^k & {\rm if}\ d\ {\rm is\ a\ product\ of}\ k\ {\rm distinct\ primes,}\\
0 & {\rm if}\ d\ {\rm is\ not\ square-free.}
\end{array}
\right.
$$
\end{dfn}

\begin{teo}
Suppose,  TBFT (resp., TBFT$_f$ or TBFT$_{ff}$) takes place for
an endomorphism $\f:G\to G$ and its iterations.
In particular, $R(\f^n)<\infty$ for any $n$.
Then one has the following Gauss congruences
for Reidemeister numbers:
 $$
 \sum_{d\mid n} \mu(d)\cdot R(\phi^{n/d}) \equiv 0 \mod n
 $$
for any $n$.
\end{teo}

\begin{proof}
This follows from \ref{cor:periodicanddyn} and the
general theory of congruences for periodic points
(cf. \cite{Smale1967BAMS,Zarelua2008Steklo}).

More precisely, let $P_n$ be the number of 
periodic points of least period $n$
of the dynamical system of \ref{cor:periodicanddyn}.
Then $R(\f^n)=\mathbf{F}(\wh\f^n)=\sum\limits_{d|n} P_d$. By the 
M\"obius inversion formula,
$$
\sum\limits_{d|n} \mu(d) R(\f^{n/d})=P_n \equiv 0
\mod n,
$$
since each orbit brings to $P_n$ just $n$ points. 
\end{proof}

\begin{rk}
In \cite{FelTroJGT} it is shown that TBFT$_f$ and
TBFT$_{ff}$ are equivalent for finitely generated groups
(see also \cite{TroTwoEx}).
\end{rk}

\section{TBFT for endomorphisms of abelian and finite groups}
\label{sec:tbftabfin}

For these classes all irreducible representations
are finite-dimensional. That is why TBFT is the same as
TBFT$_f$. In \cite{FelHill} using Pontryagin duality the following
statement was proved.

\begin{prop}\label{prop:abelcase}
TBFT$_f$ holds for any endomorphism of an abelian group.
\end{prop}

For polycyclic groups below we need also

\begin{prop}\label{prop:abelcasefin}
TBFT$_{ff}$ holds for any endomorphism of an abelian group.
\end{prop}

\begin{proof}
Indeed, in an abelian group
$$
x\f(x^{-1})y\f(y^{-1})=(xy)\f((xy)^{-1}), \qquad
(x\f(x^{-1}))^{-1}=x^{-1}\f(x),
$$
$$
xg\f(x^{-1})(yg\f(y^{-1}))^{-1}=(xy^{-1})\f((x^{-1}y)).
$$
This shows that the Reidemeister class of $e$ is a subgroup $H$,
and the other classes are $H$-cosets. Being a Reidemeister class,
$H$ is $\f$-invariant (see Lemma \ref{lem:fiinva} above) and the factorization
$p:G\to G/H$ gives a bijection of Reidemeister classes. The induced
action on $G/H$ is trivial, as well as on $\wh{G/H}$.
The fixed representations $\rho\circ p$, where $\rho$ runs
over $\wh{G/H}$, are desired finite representations.
\end{proof}

\begin{teo}[cf. \cite{FelHill}]\label{teo:peterweylforendo}
Let $\f:G\to G$ be an endomorphism of a finite group $G$.
Then the Reidemeister number $R(\f)$ coincides with
the number of $\wh\f$-{\rm\textbf{f}}-points on $\wh G$, i.e. TBFT
is true in this situation. 
\end{teo}

\begin{proof}
Let us note that $R(\f)$ is equal to the dimension of
the space of $\f$-class functions (i.e. those functions
that are constant on Reidemeister classes). They
can be also described as fixed elements of the action
$a\mapsto g a \f(g^{-1})$ on the group algebra $\C[G]$.
For the latter algebra we have the Peter-Weyl decomposition
$$
\C[G]\cong \bigoplus_{[\rho]\in \wh G} \End V_\rho,\qquad
\rho:G\to U(V_\rho),
$$
which respects the left and right $G$-actions. Hence,
$$
R(\f)=\sum\limits_{[\rho]\in \wh G} \dim T_\rho,
\qquad T_\rho:=\{a\in \End V_\rho\: |\: a=\rho(g) a \rho(\f(g^{-1})
\mbox{ for all } g\in G\}.
$$
Thus, if $0\ne a \in T_\rho$,
$a$ is an intertwining operator between the irreducible
representation $\rho$ and some representation $\rho\circ\f$.
This implies that $\rho$ is equivalent to some 
(irreducible) subrepresentation $\pi$ of $\rho\circ\f$
(cf. \cite[VI, p.57]{NaimarkStern}). Hence, 
$\dim \rho=\dim \pi$, while $\dim \rho = \dim \rho\circ\f$.
Thus, $\pi=\rho\circ\f$, and is irreducible. 
In this situation $\dim T_\rho=1$ by the Schur lemma.
Evidently, vice versa, if $\rho\sim \rho\circ \f$ then 
$\dim T_\rho =1$.
Hence,
$$
R(\f)=\sum\limits_{[\rho]\in \wh G} \left\{
\begin{array}{l}
1,\mbox{ if }\rho\sim \rho\circ \f \\
0,\mbox{ if }\rho\not\sim \rho\circ \f
\end{array}
\right.
= \mbox{ number of }\wh\f\mbox{-\textbf{f}-points}.
$$
\end{proof}

\section{Technical lemmas}\label{sec:techlem}
We will need the following

\begin{lem}\label{lem:funczional}
Let $\rho$ be a finite representation. 
It is a $\wh\f$-\textbf{f}-point if and only if
there exists a non-zero $\f$ class function
being a matrix coefficient of $\rho$.

In this situation this function is unique
up to scaling and is defined by the formula
\begin{equation}\label{eq:formulaclassfun}
T_{S,\rho}:g\mapsto \Tr (S\circ \rho(g)),
\end{equation}
where $S$ is an intertwining operator between
$\rho$ and $\rho\circ \f$:
$$
\rho(\f(x)) S= S \rho(x)\quad\mbox{ for any }x\in G.
$$

In particular, 
TBFT$_{ff}$ is true for $\f$ if and only 
if the above matrix coefficients form a base of
the space of $\f$-class functions.
\end{lem}

\begin{proof}
First, let us note that (\ref{eq:formulaclassfun})
defines a class function:
$$
T_{S,\rho}(x g \f(x^{-1}))= 
\Tr (S \rho(x g \f(x^{-1})))=
\Tr (\rho(\f(x)) S \rho(g) \rho(\f(x^{-1}))=\Tr (S\rho(g)).
$$
If $S\ne 0$, then $\rho(a)=S^*$ for some $a\in \ell^1(G)$,
and $\Tr(SS^*) \neq 0$. Thus, the $\f$-class function is non-zero.
On the other hand, any matrix coefficient of $\rho$, i.e.
a functional $T:\End(V_\rho)\to\C$ has the form
$g\mapsto \Tr(D\rho(g))$ for some fixed matrix $D\ne 0$.
If it is a $\f$-class function, then for any $g\in G$,
or similarly, $a\in \ell^1(G)$,
$$
\Tr(D\rho(a))=\Tr(D \rho(x a \f(x^{-1})))=
\Tr(\rho(\f(x^{-1})) D \rho(x) \rho(a) ).
$$
Since $\rho(a)$ runs over all the matrix algebra,
this implies $D=\rho(\f(x^{-1})) D \rho(x)$, or
$\rho(\f(x)) D= D \rho(x)$, i.e. $D$ is the desired
non-zero intertwining operator.

The uniqueness up to scaling follows now from 
the explicit formula and the Shur lemma.

The last statement follows from linear
independence of matrix coefficients of
non-equivalent representations.
\end{proof}

\begin{lem}\label{lem:factoriz}
A group $G$ satisfies TBFT$_{ff}$ for an
endomorphism $\f:G\to G$ if and only if
there exists a $\f$-equivariant factorization
of $G$ onto a finite group $F$, such that
Reidemeister classes on $G$ map onto distinct
classes on $F$.
\end{lem}

\begin{proof}
Let $\r_1$,\dots $\r_k$ be all finite
$\wh\f$-\textbf{f}-representations, $F_1$, \dots $F_k$,
$F_i=\r_i(G)$ the corresponding finite groups.
Suppose, $g\in \Ker \r_i=:K_i$. Then $\r(\f(g))=
S\r(g) S^{-1}=S S^{-1}=e$. Hence, $K_i$ is a normal
$\f$-invariant subgroup. Define $K:=\cap K_i$. It is still
a normal
$\f$-invariant subgroup of finite index.
Let $F:=G/K$. By Lemma \ref{lem:epimorphmappingofclasses}
the Reidemeister number of the induced map $\f_F$ 
satisfies $R(\f_F)\leq R(\f)$.
On the other hand, $\r_i$ define disjoint irreducible representations of $F$ corresponding to  $\wh\f_F$-{\rm\textbf{f}}-points on $\wh F$. By Lemma \ref{lem:funczional}
this implies $R(\f_F)\geq R(\f)$. Thus, $G\to F$ gives a bijection
of Reidemeister classes.

The opposite statement follows from Lemma \ref{lem:funczional}
and Theorem \ref{teo:peterweylforendo}. Indeed, if $p:G\to F$
is the mentioned epimorphism with $r=R(\f)=R(\f_F)$, then by  
Theorem \ref{teo:peterweylforendo} we have exactly 
$r$ pairwise non-equivalent irreducible unitary 
$\wh{\f_F}$-\textbf{f}-representations
$\r_1,\dots,\r_r$ of $F$. Then $\r_i\circ p$ play the same role
for $G$. Finally $G$ can not have ``additional'' $\wh\f$-\textbf{f}-representations by linear independence in Lemma 
\ref{lem:funczional}.
\end{proof}
We can enforce this statement.

\begin{lem}\label{lem:nonequivfactoriz}
Let $\f:G\to G$ be an endomorphism with $R(\f)<\infty$.
Suppose there exists a (not necessarily equivariant) factorization
$p$
of $G$ onto a finite group $F$, such that
Reidemeister classes on $G$ map bijectively onto some disjoint
subsets (not necessarily classes) on $F$.
Then $G$ has TBFT$_{ff}$ for $\f$.
\end{lem}

\begin{proof}
For $F$, characteristic functions of \emph{any} sets are
matrix coefficients of some (finite) representations.
Thus, characteristic functions of Reidemeister classes on $G$
are matrix coefficients of some finite representations
(coming from $F$). Lemma \ref{lem:funczional} completes
the proof.
\end{proof}

\begin{dfn}
We will say under the supposition of Lemma \ref{lem:nonequivfactoriz}
that $p$
\emph{separates} Reidemeister classes. Similarly, we will say
that $p$ separates some classes if it maps them to disjoint
subsets.
\end{dfn}

\begin{lem}\label{lem:shiftsfinite}
Let $\f$ be an endomorphism of a 
group $G$ with $R(\f)<\infty$. Then $G$ has 
TBFT$_{ff}$ for $\f$
if and only if the right shifts
of all Reidemeister classes form a finite number of subsets
of $G$.
\end{lem}

\begin{proof}
First, let us note, that if we have a bijection for  Reidemeister
classes of $\f$, then we have a bijection for Reidemeister
classes of any $\t_g\circ \f$ by Lemma \ref{lem:shifts}.

Now the ``only if'' direction is evident, because these sets are
pre-images of some sets in finite $F$.

Suppose, there are finitely many shifts. This means, that the
stabilizers under right shifts of any Reidemeister class of $\f$
have finite index. Since $R(\f)<\infty$, their intersection
is a subgroup $H$ in $G$ of finite index. 
By Lemma \ref{lem:shifts}, its elements stabilize Reidemeister
classes of any $\t_g\circ\f$.

Suppose, $h\in H$, then
$$
yg\f(y^{-1})z h z^{-1}=y(gz) z^{-1}\f(y^{-1})z h z^{-1}=
y(gz)(\t_{z^{-1}}\circ \f)(y^{-1}) h z^{-1}
$$
$$
=
x(gz)(\t_{z^{-1}}\circ \f)(x^{-1}) z^{-1}=
xg\f(x^{-1}) 
$$
for some $x$.
Thus, $H$ is normal (and $\f$-invariant, if $\f$ is an automorphism).
The projection $G\to G/H=:F$ maps Reidemeister classes of $\f$
bijectively onto disjoint sets in $F$. Indeed, if $h\in H$, then for any
$g$,
$$
e\cdot h= x e (\t_g\circ \f)(x^{-1}).
$$
This means, that $H$ entirely is in the Reidemeister class
of $e$ of any $\t_g\circ\f$. By Lemma \ref{lem:shifts}
this means that all Reidemeister classes of $\phi$ are formed
by $H$-cosets and we are done.
 
It remains to apply Lemma \ref{lem:nonequivfactoriz}.
\end{proof}

\begin{lem}\label{lem:numberofinner}
Suppose, $G$ has only finitely many inner automorphisms
and an endomorphism $\f:G\to G$ with $R(\f)<\infty$.
Then $G$ has TBFT$_{ff}$ for $\f$.
\end{lem}

\begin{proof}
By Lemma \ref{lem:shifts}, the number of right shifts of Reidemeister classes
is not more than $R(\phi)\cdot |Inn(G)|$. It remains to
apply Lemma \ref{lem:shiftsfinite}.
\end{proof}

\begin{lem}\label{lem:fixedpointsforabelian}
Let $\f:A\to A$ be an endomorphism of a finitely
generated abelian group $A$ with $R(\f)<\infty$.
Then the number of fixed elements $C(\f)$ on $A$
is finite.
\end{lem}

\begin{proof}
The torsion subgroup is finite and totally invariant.
Factorizing we reduce the problem to the case of $\Z^n$
by Lemma \ref{lem:epimorphmappingofclasses}.
In this case we will show that $\f$ has only the trivial
fixed element $0$. By \cite{FelHill} in this case
$\det(\Id-\f)\ne 0$, considered as $n\times n$ integer matrix.
Diagonalising this matrix by left and right multiplication
by unimodular matrix (as it was done e.g. in \cite{BBPT})
we see that it can not have (non-zero) eigenvector with eigenvalue 
zero, i.e. there is no non-trivial $\f$-fixed point. 
\end{proof}

\section{TBFT$_{ff}$ for endomorphisms of polyciclic groups}
\label{sec:tbftpolyc}

Consider a polycyclic group $G$. Its (finite) derived series is
formed by fully invariant subgroups $G_i$ with abelian
quotients $A_i=G_i/G_{i+1}$. The key difference from a
general finitely generated solvable group is the following:
all $G_i$ and $A_i$ are finitely generated (see \cite{Robinson} for
details). Let $G_n\neq \{e\}$, $G_{n+1}=\{e\}$.

We will argue by induction. For the basis of this
induction let us observe that for the abelian group $G_n$
and any its endomorphism with finite Reidemeister number
we have TBFT$_{ff}$ by Prop. \ref{prop:abelcase}.
Now suppose by induction that same is true for $G_{i+1}$
and prove it for $G_i$.
Denote $G_{i+1}=:H$, $G_i=:\G$, $A_i=:A$.

Consider an endomorphism $\f:\G\to\G$ with $R(\f)<\infty$
and induced endomorphisms $\f_H:H\to H$ and $\f_A:A\to A$.
Then by Lemma \ref{lem:epimorphmappingofclasses}
$R(\f_A)<\infty$ and by Lemma \ref{lem:fixedpointsforabelian}
the number of fixed elements of $\f_A$ on $A$ is finite.
Thus, by Lemma \ref{lem:fixedpointsonfaxtorandreid}
$R(\f_H)<\infty$. Let $H_0\ss H$ be a normal $\f$-invariant
subgroup of finite index such that $H\to H/H_0$ gives a 
bijection on Reidemeister classes (see Lemma \ref{lem:factoriz}).
This means that classes of $\f_H$ are some unions of
$H_0$ cosets. Consider automorphisms $\t_g: H\to H$,
$\t_g(h)=ghg^{-1}$,
$g\in \G$, and define $H_1:=\cap_{g\in \G} \t_g(H_0)$.
All subgroups in the intersection have the same finite index
in $H$. Since $H$ is finitely generated, $H_1$ also has
a finite index. By construction, $H_1$ is normal in $\G$.
Also, 
$$
\f(\t_g(h_0))=\f(g) \f(h_0) (\f(g))^{-1}=
\t_{\f(g)}(\f(h_0)),\qquad h_0,\: \f(h_0) \in H_0,
$$
i.e. $\f(\t_g(H_0))\ss \t_{\f(g)}(H_0)$. Hence, $H_1$
is $\f$-invariant and $H_1\ss H_0$, thus
$H_1$ can play the same role as $H_0$
with an additional property of being normal in $\G$.
In particular,  classes of $\f_H$ are some unions of
$H_1$ cosets. The same is true for intersections
of Reidemeister classes of $\f$ with $H$ (because
they are unions of some classes of $\f_H$). This means
that $\G\to \G/H_1$ separates these classes, i.e. the classes
which map on the Reidemeister class of $e\in A$ under
$\G\to A$. 
Similarly we can find subgroups $H_2,\dots H_{R(\f_A)}$
which separate classes over other classes (i.e. which are mapped onto other classes)
of $A$.
For this purpose, suppose $g\in \G$ is over some other
class of $A$. Then the classes of $\t_g\circ \f$ 
(with the same finite Reidemeister number) that intersect
with $H$ are obtained from the classes under consideration
by a shift by $g$ (see Lemma \ref{lem:shifts}). Then a group $H_2=H^g_1$ constructed
for $\t_g\circ \f|_H$ in the same way as $H_1$ for $\f_H$,
will separate these classes. 
Moreover, it will be $\t_g \circ \f$-invariant for this
specific $g$ and normal in $\G$, i.e. $\t_{g'}$-invariant
for any $g'\in \G$. Taking $g'=g^{-1}$ we see that
$H^g_1$ is $\f$-invariant. 
Similarly we define the other $H_i$
choosing $g$ over other Reidemeister classes of $\f_A$.
Taking the (finite !) intersection
of $H_1$, $H_2$,... $H_{R(\f_A)}$ we obtain a $\f$-invariant 
subgroup of finite index
$H'\ss H$, which is normal in $\G$, and
$p: \G\to \G/H'=\G'$ gives a bijection on Reidemeister
classes; $F:=p(H)=H/H' \ss \G'$ is a finite normal subgroup
and $\G'/F\cong A$:
$$
\xymatrix{
H\ar@{^{(}->}[r]\ar[d]^p&  \G\ar[d]^{p} \ar[rd] &\\
H/H'\ar@{^{(}->}[r]\ar@{=}[d]&  \G/H'\ar@{=}[d] \ar[r] & \G/H\ar@{=}[d]\\
F\ar@{^{(}->}[r]& \G' \ar[r]&A.
}
$$
Thus, $\G'$ is finitely generated finite-by-abelian group, 
and it has finitely many inner automorphisms. Lemma \ref{lem:numberofinner} completes the proof of the
following statement.

\begin{teo}\label{teo:polyend}
Let $\f:G\to G$ be an endomorphism of a polycyclic
group with $R(\f)<\infty$. Then TBFT$_{ff}$ is true
for $\f$.
\end{teo}

\begin{rk}
The results of this section can be extended to some
virtually polycyclic groups (under supposition of some
polycyclic subgroup to be $\f$-invariant).
\end{rk} 

More precisely the following statement can be proved by word-by-word
rewriting of the above argument.

\begin{teo}\label{teo:almostpolyend}
Let $G$ be an almost polycyclic group admitting a fully invariant
polycyclic subgroup of finite index. Then TBFT$_{ff}$ is true
for any endomorphism $\f:G\to G$ with $R(\phi)<\infty$. 
\end{teo}

Let us illustrate this theorem by the following
\begin{ex}\label{examp:halee}
In \cite{HaLee2015} seven series of almost polycyclic (polycyclic-by-finite)
groups with the property $R_\infty$ (any automorphism has infinite
Reidemeister number) were found. They have a fully invariant  
polycyclic subgroup of finite index (\cite[Remark 2.1]{HaLee2015},
\cite{LeeLee2009}).
Thus, they are covered by Theorem \ref{teo:almostpolyend}.

On the other hand, they evidently have endomorphisms with finite
Reidemeister number (see Prop. \ref{prop:drast} and Cor.
\ref{cor:finiteimendo}).
\end{ex}


\def\cprime{$'$} \def\cprime{$'$} \def\cprime{$'$} \def\cprime{$'$}
  \def\cprime{$'$} \def\cprime{$'$} \def\cprime{$'$} \def\cprime{$'$}
  \def\dbar{\leavevmode\hbox to 0pt{\hskip.2ex \accent"16\hss}d}
  \def\cprime{$'$} \def\cprime{$'$}
  \def\polhk#1{\setbox0=\hbox{#1}{\ooalign{\hidewidth
  \lower1.5ex\hbox{`}\hidewidth\crcr\unhbox0}}} \def\cprime{$'$}
  \def\cprime{$'$} \def\cprime{$'$} \def\cprime{$'$}

\end{document}